\documentclass{article}
\usepackage[utf8]{inputenc}
\usepackage{amssymb}
\usepackage{dsfont}
\usepackage{textcomp}
\usepackage{setspace}
\usepackage[english]{babel}
\usepackage{graphicx}
\usepackage{authblk}
\graphicspath{{./images/}}
\usepackage{amsmath, amssymb, amsthm}
\usepackage{booktabs}
\usepackage{siunitx}
\usepackage{graphicx}
\usepackage{subcaption}
\usepackage{amsmath}
\usepackage{algorithm2e}
\usepackage[a4paper, total={15cm, 22cm}, left=3cm, right=3cm, top=2cm, bottom=2cm]{geometry}
\newtheorem{theorem}{\sc Theorem}[section]

\newtheorem{dft}[theorem]{\sc Definition}
\newtheorem{lem}[theorem]{\sc Lemma}
\newtheorem{prop}[theorem]{\sc Proposition}
\newtheorem{cor}[theorem]{\sc Corollary}
\newtheorem{rem}[theorem]{\sc Remark}
\newtheorem{ex}[theorem]{\sc Example}

\usepackage[
    colorlinks=true,    
    linkcolor=blue,     
    citecolor=green, 
    filecolor=Magenta,  
    urlcolor=black   
]{hyperref}
\usepackage[nottoc]{tocbibind}
\usepackage{cite}  

\begin{document}

\title{Proportion of Simple Subgroups in Finite Groups and Their Applications}
\author{João Victor Monteiros de Andrade  \thanks{Computer Science Department, University of Brasília; jotandrade98@gmail.com } \and Leonardo Santos da Cruz \thanks{Computer Science Department, University of Brasília; leonardo-7238@hotmail.com }}

\maketitle

\section{Abstract}

This work introduces and investigates the function \( \mathcal{V}(G) = \frac{\text{Simp}(G)}{|L(G)|} \), where \( \text{Simp}(G) \) denotes the number of simple subgroups and \( |L(G)| \) the total number of subgroups of a finite group \( G \). The function \( \mathcal{V}(G) \), defined on the interval \( [0,1] \), represents the proportion of simple subgroups relative to the total number of subgroups. It serves as a tool for analyzing structural patterns in finite groups, particularly in p-groups and other families.  \\  

\textbf{Keywords:} GAP; functions, simple groups.

\section{Introduction}

Functions defined from groups are a topic of ongoing development that has been yielding promising results. Some cases have already been explored in \cite{Marius}, \cite{TM}, and \cite{deandrade2025}. Many of these functions have properties that have been previously studied, such as multiplicativity. The functions considered in \cite{Marius}, \cite{TM}, and \cite{deandrade2025} possess this property, and it is repeatedly used in the verification and application of important results. A natural question that arises is whether, even in the absence of this property, it is still possible to obtain consistent results in terms of asymptotic behavior and group classification. To this end, we consider the function 

\begin{dft}
$\mathcal{V}(G) = \frac{Simp(G)}{|L(G)|}$, where $Simp(G)$ denotes the number of simple subgroups of G.
\end{dft}

For simplicity, we denote the quantity $Simp(G)$ by $s(G)$. Note that this function is, in general, not multiplicative, since for a large class of groups—for example, solvable groups—the following result holds.

\begin{prop}
Let $G$ be a finite solvable group that has the following decomposition $G = \otimes_{i=1}^{n} H_{i}$, where $gcd(H_{i},H_{j}) = 1$, for $i\neq j$. Then

\begin{equation}\label{eq. 1.1}
    \mathcal{V}(G) = \frac{\sum_{i=1}^{n}s(H_{i})}{\prod_{i=1}^{n}|L(H_{i})|}.
\end{equation}
\end{prop}
    
\begin{proof}
The denominator expression is given by Theorem 1.1 in \cite{Betz16032022}. Thus, it suffices to define the numerator. If $G$ is solvable, the only simple groups that $G$ admits are cyclic subgroups of prime order. Thus, $s(H_1\times H_2)=s(H_1)+s(H_2),$ since a prime-order element in \(H_1\times H_2\) has the form
\((g,e)\) or \((e,h)\). By induction on the index, the result follows.   
\end{proof}

\begin{ex}
\begin{align*}
     \mathcal{V}(C_3 \times C_5) &= \frac{1+1}{2\cdot2} =\frac{1}{2};\\   \mathcal{V}(C_3) \cdot  \mathcal{V}(C_5) &= \frac{1}{2} \cdot \frac{1}{2} = \frac{1}{4}. 
\end{align*}    
\end{ex}

It is worth noting that the expression \ref{eq. 1.1} can, with the necessary adaptations, be used in the analysis of other families of groups that are not necessarily solvable.

\begin{prop}\label{prop. 2.4}
Let $G = C_{p^n} \times A_{5}$, where $p$ is prime and $p > 5$, then $$\mathcal{V}(G) = \frac{33}{59(n+1)}.$$ 
\end{prop}

\begin{proof}
Note that for $p>5$, $gcd(|C_{p^n}|, |A_5|) =1$, then using \ref{eq. 1.1} we have $s(C_{p^n}) = 1$ and $s(A_5) = 32$. For the denominator we have $|L(C_{p^n})| = (n+1)$ and $|L(A_5)| = 59$.

\end{proof}

It is still possible to obtain consistent results for this type of function even if it is not multiplicative, as can be seen in the following propositions for subfamilies of solvable groups.

\begin{prop}
Let $G$ be a cyclic group of order $n = \prod_{i=1}^{m} p_i^{a_i}$, 
$p_1, p_2, \dots, p_m$ distinct primes and $a_1, a_2, \dots, a_m \ge 1$, then

\begin{equation}
 \mathcal{V}(G) = \frac{\omega(n)}{\tau(n)} = \frac{m}{\prod_{i=1}^{m} (a_i + 1)},  
\end{equation}

\noindent{where $\omega$ indicates the number of distinct prime factors of n and $\tau$ indicates the number of divisors of n.}  
\end{prop}

\begin{proof}
Every subgroup of a cyclic group $C_n$ is cyclic, and for each divisor $d$ of $n$, there exists exactly one subgroup of order $d$.
The simple subgroups of $C_n$ are exactly the subgroups of order p with p prime and p divides n. Thus, the number of simple subgroups of $C_n$ coincides with the number of distinct primes that divide $n$. The total number of subgroups is given by $\tau$. Performing the division gives the result.   
\end{proof}

\begin{cor}
Let $G = H \times A_{5}$, where $H$ is a cyclic group with order $n = \prod_{i=1}^{m} p_i^{a_i}$, 
$p_1, p_2, \dots, p_m$ distinct primes greater than 5 and $a_1, a_2, \dots, a_m \ge 1$, then

\begin{equation*}
  \mathcal{V}(G) = \frac{32+ \omega(n)}{59\cdot\tau(n)}    
\end{equation*}

\end{cor}

\begin{prop}
  Let $C_{p}^n = C_p \times ...\times C_p$, then

\begin{equation}
 \mathcal{V}(C_{p}^{n}) = \frac{ \binom{n}{1}_p}{
\sum_{k=0}^{n}\binom{n}{k}_p}, 
\end{equation}

\noindent{where $\binom{n}{k}_p = \prod_{i=1}^{k} \frac{p^n - p^{i-1}}{p^k - p^{i-1}}$ and $\binom{n}{1}_p =\frac{p^n-1}{p-1}.$}
\end{prop}

\begin{proof}
The group $C_p^{n}=C_p\times\cdots\times C_p$ (with $p$ - prime) is isomorphic to the vector space $\mathbb{F}(p^n)$; thus, the subgroups correspond exactly to the vector subspaces. The number of subgroups of dimension $k$ (i.e., of order $p^k$) is given by the Gaussian binomial coefficient (or $q$-binomial) with $q = p$:

$$
\binom{n}{k}_p = \frac{\prod_{i=0}^{k-1} (p^n - p^i)}{\prod_{i=0}^{k-1} (p^k - p^i)}.
$$

The total number of subgroups of $C_p^{n}$ is the sum of these terms for $(k=0,\dots,n):$ $\sum_{k=0}^{n}\binom{n}{k}_{p}.$ Now, let $G=C_p^n\cong\mathbb{F}_p^n$. A subgroup of $G$ is simple (as an abelian group) if and only if it is cyclic of prime order; therefore, the simple subgroups of $G$ are exactly the subgroups of order $p$, i.e., the 1-dimensional subspaces of $\mathbb{F}_p^n$. We count these subspaces as follows: there are $p^n-1$ non-zero vectors in $\mathbb{F}_p^n$, and each 1-dimensional subspace contains exactly $p-1$ non-zero vectors, so the number of such subspaces is

$$
\frac{p^n-1}{p-1}=1+p+\cdots+p^{,n-1} = \binom{n}{1}_p.
$$
\end{proof}

\begin{cor}
    $\lim_{n \longrightarrow \infty } \mathcal{V}(C_{p}^{n}) = 0.$
\end{cor}

\begin{cor}
   Let $C_{p}^2$, then $ \mathcal{V}(C_{p}^{2}) = \frac{p+1}{p+3}.$
\end{cor}


\begin{cor}
    $\lim_{p \longrightarrow \infty } \mathcal{V}(C_{p}^{2}) = 1.$
\end{cor}

\begin{cor}
Let $G$ be a non-cyclic abelian group. Then for some p and $n \geq 5$
\begin{equation*}
     \mathcal{V}(C_{p}^{n}) \leq  \mathcal{V}(G) \leq  \mathcal{V}(C_{p}^{2}).
\end{equation*}
\end{cor}

\subsection{Families of finite $p$-groups}

\begin{prop}
Consider    
\begin{itemize}
    \item{$D_{2^n} = \bigl\langle\,x,y \bigm| x^{2^{n-1}} = 1,\; y^2 = 1,\; y\,x\,y = x^{-1} \bigr\rangle$;} 
    \item{$Q_{2^n}=\langle x,y \mid x^{2^{n-1}}=1,\; y^2=x^{2^{n-2}},\; yxy^{-1}=x^{-1}\rangle, \, (n\ge 3);$} 
    \item{$QD_{2^n} = \langle x, y \mid x^{2^{n-1}} = 1,\; y^2 = 1,\; yxy = x^{r} \rangle, \, (n \geq 4)$} 
\end{itemize}

then

\begin{enumerate}
    \item{$ \mathcal{V}(D_{2^n}) = \frac{2^{n-1}+1}{2^n +n-1}$;} 
    \item{$ \mathcal{V}(Q_{2^{n}}) = \frac{1}{2^{n-1}+n-1}$;} 
    \item{$ \mathcal{V}(QD_{2^n}) = \frac{2^{n-2}+1}{3 \cdot2^{n-2} +n-1}$.} 
\end{enumerate}

\end{prop}

\begin{proof}

In these cases, since the groups are p-groups (in particular 2-groups), the simple subgroups will be cyclic of prime order. Since the only prime to be considered is 2, it is sufficient to count the number of involutions in each group. 

\begin{enumerate}
    \item Consider the dihedral group $D_{2^n} = \langle x,y \mid x^{2^{n-1}}=1,\; y^2=1,\; yxy=x^{-1}\rangle$. Its elements are the rotations \(x^m\) and the reflections \(y x^k\), with \(0\le k<2^{n-1}\). The rotations \(x^m\) have a divisor order of \(2^{n-1}\). The only rotation of order \(2\) is \(x^{2^{n-2}}\), since $(x^{2^{n-2}})^2 = x^{2^{n-1}} = 1.$ On the other hand, for any \(k\), the element \(y x^k\) is a reflection and satisfies $(y x^k)^2 = y x^k y x^k = 1$, therefore all reflections are involutions. Since there are \(2^{n-1}\) reflections, it follows that the total number of involutions in \(D_{2^n}\) is $2^{n-1} + 1.$ It is concluded, therefore, that the number of subgroups of order \(2\) in \(D_{2^n}\) is $2^{n-1} + 1.$
    \item Consider the generalized quaternary group $Q_{2^n}=\langle x,y \mid x^{2^{n-1}}=1,\; y^2=x^{2^{n-2}},\; yxy^{-1}=x^{-1}\rangle,$ with \(n\ge 3\). The element \(x^{2^{n-2}}\) is central and satisfies $(x^{2^{n-2}})^2 = x^{2^{n-1}} = 1,$ and is therefore an involution. If \(x^k\) is a power of \(x\) distinct from \(x^{2^{n-2}}\), then its order is greater than \(2\). Furthermore, for any element of the form \(x^k y\), we have $(x^k y)^2 = x^k y x^k y = x^k (y x^k y^{-1}) y^2 = x^k x^{-k} x^{2^{n-2}} = x^{2^{n-2}} \neq 1.$ Therefore, no element of the form \(x^k y\) is an involution. Thus, \(Q_{2^n}\) has exactly one involution, and consequently a unique subgroup of order \(2\).
    \item Consider the semidihedral group $ QD_{2^n}=\langle x,y \mid x^{2^{n-1}}=1,\; y^2=1,\; yxy=x^{r}\rangle, \quad \text{where } r=2^{n-2}-1,$ with \(n\ge 4\). The rotation \(x^{2^{n-2}}\) satisfies $(x^{2^{n-2}})^2 = x^{2^{n-1}} = 1,$ therefore it is an involution. For the elements of the form \(y x^k\), we have $ (y x^k)^2 = y x^k y x^k = (y x^k y) x^k = x^{r k} x^k = x^{(r+1)k}.$ Since \(r+1 = 2^{n-2}\), it follows that $(y x^k)^2 = x^{2^{n-2}k}.$ This element is equal to the identity if, and only if, \(k\) is even modulo \(2^{n-1}\). There are exactly \(2^{n-2}\) such values of \(k\), producing \(2^{n-2}\) involutions of the form \(y x^{2m}\). Therefore, the total number of involutions in \(QD_{2^n}\) is
$2^{n-2} + 1,$ and, consequently, the number of subgroups of order \(2\) is $2^{n-2}+1.$
\end{enumerate}

Dividing these quantities by the respective total quantities of subgroups within each group (see \cite{Marius}) yields the result.  
\end{proof}

In general, for 2-groups, the following result is obtained.

\begin{prop}
   Let G be a finite non-abelian 2-group, then
\begin{equation*}
    \frac{\mathcal{V}(G)}{|G|} < \frac{1}{2\sqrt{10}}. 
\end{equation*}
\end{prop}

\begin{proof}
 From \cite{It1951}, we have $(t+1)^2 \leq k(G)|G|$, where $(t+1)$ denotes the number of involutions in $G$ and $k(G)$ denotes the number of conjugacy classes of $G$. Hence it follows

\begin{align*}
   (t+1)^2 \leq k(G)|G| \iff \frac{(t+1)^2}{|G|^2} &\leq \frac{k(G)}{|G|} = cp(G).\\
\end{align*}

By \cite{Gustafson01111973} we have

\begin{align*}
   (t+1)^2 \leq k(G)|G| \iff \frac{(t+1)^2}{|G|^2} &\leq  \frac{5}{8}\\
    \frac{(t+1)}{|G|} &\leq \sqrt{\frac{5}{8}}\\
    \frac{(t+1)}{|G|}\frac{|L(G)|}{|L(G)|} &\leq \sqrt{\frac{5}{8}}\\
    \mathcal{V}(G)\beta(G) &\leq \sqrt{\frac{5}{8}}\\
    \mathcal{V}(G)&\leq \frac{1}{\beta(G)}\sqrt{\frac{5}{8}}, \\
\end{align*}

\noindent{where $\beta$ is defined in \cite{Lazorec}. By the contrapositive of Proposition 2.2.2 in \cite{Lazorec}, we have}

\begin{align*}
    \mathcal{V}(G)&< \frac{|G|}{5}\sqrt{\frac{5}{8}} \iff
    \frac{\mathcal{V}(G)}{|G|}< \frac{1}{2\sqrt{10}}. 
\end{align*}

\end{proof}

\begin{dft} \label{Def. 1}
   Let $p$ be an odd prime. We recursively define families of finite extraspecial groups.
\[
\{G_n\}_{n \ge 0}, \qquad  \{H_n\}_{n \ge 0}, 
\]
as follows:

\begin{itemize}
    \item $G_0 := C_{p^2}  \rtimes C_p,
\quad
G_{n+1} := (C_p \times G_n) \rtimes C_p;$
    \item $H_0 := (C_p \times C_p) \rtimes C_p,
\quad
H_{n+1} := (C_p \times H_n) \rtimes C_p.$
\end{itemize}

\noindent{Where, at each iteration, the semidirect product is defined in a manner analogous to Heisenberg-type constructions.}
\end{dft}

\begin{prop}
For $n\geq 1$ we have
    $\lim_{p \longrightarrow \infty} \mathcal{V}(G_{n}) = \lim_{p \longrightarrow \infty}\mathcal{V}(H_{n}) =0.$ 
\end{prop}

\begin{proof}
We will carry out the proof for $G_{n+1}$, observing that it is analogous for $H_{n+1}$ after making the necessary adaptations. First, it is necessary to define $s(G_{n+1})$. Since $G_n$ is a $p$-group, there are no cyclic subgroups of prime order different from $p$; it suffices to count the subgroups of order $p$. Denote $\Omega_1(G_n)=\{g \in G_n : g^p = 1\}$. We want $|\Omega_1(G_n)|$. Structurally, $G_n$ is a class-2 group obtained from one generator $a$ of order $p^2$ and several generators of order $p$. The calculation of invariants yields

   \[
   Z(G_n)\cong C_p^{n+1},\qquad G_n/Z(G_n)\cong C_p^{n+2},
   \]

\noindent{therefore, the $p$-torsion part has $p$-rank $2n+2$. Denote any element as $a^\alpha \cdot u$ with $\alpha \in \{0,\dots,p^2-1\}$ and $u$ in a certain substructure $U$ of exponent $p$ and order $p^{2n+1}$ (generated by the generators of order $p$ coming from the $C_p$ factors and by the central generators of order $p$). Note that $U$ has exponent $p$, hence $u^p=1$ for every $u \in U$. For class-2 groups and odd $p$, the power identity holds in which the mixed terms involve binomial coefficients $\binom{p}{k}$ with $0<k<p$; these coefficients are multiples of $p$. Since the commutators here have order $p$, these mixed terms vanish. Thus}

$$
   (a^\alpha u)^p = a^{p\alpha}.
$$

Hence $(a^\alpha u)^p = 1$ if and only if $a^{p\alpha} = 1$, that is, $p \mid \alpha$.
Therefore, there are exactly $p$ choices for $\alpha$ satisfying $p \mid \alpha$. Thus

\[
|\Omega_1(G_n)| = p \cdot p^{2n+1} = p^{2n+2}.
\]

Removing the identity, the number of elements of order exactly $p$ is $p^{2n+2} - 1$. Each cyclic subgroup of order $p$ has $p-1$ generators, hence the number of cyclic subgroups of prime order is

\begin{equation}\label{eq. 3}
   \frac{p^{2n+2}-1}{p-1}. 
\end{equation}
   
Now it is necessary to define $|L(G_{n+1})|$. Determining this factor in general is a complex task. The strategy in this case will be to employ estimates for $|L(G_{n+1})|$. A lower bound, which we denote by $L_{-}$, is $p^{\frac{(n+1)^2}{4}}$, while an upper bound ($L_{+}$) is $(2n+3)p^{(2n+3)(2n+2)}$. Indeed, for $L_{-}$, in \(G_0 = C_{p^2} \rtimes C_p\), the subgroup of order \(p\) of \(C_{p^2}\) is characteristic in \(C_{p^2}\) and therefore normal in \(G_0\).
By the recursive construction of $G_{n+1} = (C_p \times G_n) \rtimes C_p$, the factors of order \(p\) introduced at each step remain normal. Thus, in \(G_n\) there exists a normal subgroup

\[
V\cong C_p^{\,n+1},
\qquad |V|=p^{\,n+1},
\]

generated by the subgroup of order \(p\) of \(C_{p^2}\) together with the \(n\) factors \(C_p\) added throughout the construction.
Viewed as a vector space over \(\mathbb{F}_p\) of dimension \(n+1\), the number of subgroups of \(V\) coincides with the number of its subspaces:
\[
|L(V)|
=\sum_{k=0}^{n+1}\begin{bmatrix} n+1 \\ k \end{bmatrix}_p,
\]
where \(\begin{bmatrix} n+1 \\ k \end{bmatrix}_p\) denotes the Gaussian binomial coefficient \cite{He2022SomeIO}.
A standard estimate is
\[
\begin{bmatrix} n+1 \\ k \end{bmatrix}_p \ge p^{\,k(n+1-k)}.
\]

Choosing \(k=\left\lfloor \frac{n+1}{2}\right\rfloor\), where the term is maximal, one obtains, for some constant \(C=O(n)\),
\[
\begin{bmatrix} n+1 \\ \left\lfloor \frac{n+1}{2}\right\rfloor \end{bmatrix}_p
\ge p^{\frac{(n+1)^2}{4}-Cn}.
\]
Ignoring the linear term in the exponent, we obtain the lower bound
\[
|L(G_n)| \ge |L(V)| \ge p^{\frac{(n+1)^2}{4}}.
\]

For $L_{+}$, consider the number of minimal generators, which we denote by $d(G_0)$. In general, $d(G_0)=2$ and $d(G_{n+1})=d(G_n)+2$, hence $d(G_n)=2n+2.$
Every subgroup is generated by a number of elements less than or equal to $d(G_n)$, hence

\[
|L(G_n)|\le\sum_{k=0}^{d(G_n)}|G_n|^k \le (d(G_n)+1)|G_n|^{d(G_n)}.
\]
Substituting $d(G_n)=2n+2$ and $|G_n|=p^{2n+3}$, one obtains
\[
|L(G_n)| \le (2n+3)p^{(2n+3)(2n+2)}.
\]
Now, taking the order relation
\[
\mathcal{V}(G_n)_{-} \le \mathcal{V}(G_n) \le \mathcal{V}(G_n)_{+},
\]
\noindent where $\mathcal{V}(G_n)_{-} = \frac{p^{2n+2}-1}{(p-1)L_{-}}$ and $\mathcal{V}(G_n)_{+} = \frac{p^{2n+2}-1}{(p-1)L_{+}}$, and taking the limit, the result follows.
\end{proof}





\subsection{Supersolvable cases}

\begin{prop}
   Let $\mathrm{Dic}_n=\langle a,x \mid a^{2n}=1,\, x^2=a^n,\, x^{-1}ax=a^{-1}\rangle$ be a dicyclic group with order 4n, then

\begin{equation*}
      \mathcal{V}(\mathrm{Dic}_n) = \frac{\omega(2n)}{\tau(2n)+\sigma(n)}.
\end{equation*}    

\end{prop}

\begin{proof}
Let $\mathrm{Dic}_n=\langle a,\,x \mid a^{2n}=1,\, x^2=a^n,\,x^{-1}ax=a^{-1}\rangle$ be the dicyclic group of order $4n$. The subgroup $\langle a\rangle$ is cyclic of order $2n$; by a basic property of cyclic groups, for each divisor $d$ of $2n$ there exists exactly one subgroup of $\langle a\rangle$ of order $d$. In particular, for each prime $p$ that divides $2n$ there exists exactly one subgroup of order $p$, generated by $a^{2n/p}$. Every element of $\mathrm{Dic}_n$ that does not belong to $\langle a\rangle$ can be written in the form $a^k x$; Using the conjugation relation, we obtain $(a^k x)^2=a^k x a^k x=a^k a^{-k} x^2=a^n$. Since $a^n$ has order 2 (because $(a^n)^2=a^{2n}=1$ and $(a^n\neq1)$), it follows that each element of the form $a^k x$ has order 4 (and, in particular, cannot generate any subgroup of prime order greater than 2). Furthermore, the only involution of $\langle a\rangle$ is $a^n$, and there are no involutions outside of $\langle a\rangle$ because, as we have seen, for $g=a^k x$ we have $g^2=a^n\neq1$. Therefore, the only subgroup of order 2 is $\langle a^n\rangle$. It follows that all prime subgroups of $\mathrm{Dic}_n$ are exactly the subgroups of $\langle a\rangle$ corresponding to the primes that divide $2n$; consequently, the number of simple subgroups (that is, of prime order) of $\mathrm{Dic}_n$ is equal to the number of distinct primes that divide $2n$. In arithmetic notation, if $\omega(m)$ denotes the number of distinct primes that divide $m$, then $s(\mathrm{Dic}_n)=\omega(2n).$
\end{proof}

\begin{prop}
    Let $G = D_{2n}$, then
    
\begin{equation*}
  \mathcal{V}(G) = \frac{n + \omega(n)}{\tau(n) \;+\; \sigma(n)}. 
\end{equation*}

\end{prop}

\begin{proof}
In this case, it suffices to justify the numerator. Let \(D_{2n}=\langle r,s \mid r^{n}=s^{2}=1,\; srs=r^{-1}\rangle\) be the dihedral group of order \(2n\).
A simple subgroup of \(D_{2n}\) necessarily has prime order.
The reflections of \(D_{2n}\) are the elements of the form \(r^{k}s\), with \(k=0,\dots,n-1\), and satisfy \((r^{k}s)^2=1\). Thus, each reflection generates a subgroup of order \(2\), and since there are exactly \(n\) reflections, we obtain \(n\) distinct simple subgroups of this form.

If \(p>2\) is prime and \(H\leq D_{2n}\) is a subgroup of order \(p\), then \(H\) cannot contain reflections, since these have order \(2\). Therefore \(H\subseteq \langle r\rangle\), where \(\langle r\rangle\) is a cyclic subgroup of order \(n\). In a cyclic group there is exactly one subgroup of order \(p\) for each prime \(p\) that divides \(n\); therefore, for each such prime we obtain a unique simple subgroup of \(D_{2n}\) contained in \(\langle r\rangle\).
These subgroups are distinct from those generated by reflections. Consequently, the total number of simple subgroups of \(D_{2n}\) is the sum of the number of reflections, \(n\), with the number of distinct primes that divide \(n\), denoted by \(\omega(n)\). Thus, \(D_{2n}\) has exactly \(n+\omega(n)\) simple subgroups.
\end{proof}

\begin{cor} For n = p, prime follows
    $\lim_{p \longrightarrow \infty } \mathcal{V}(D_{2p}) = 1.$
\end{cor}

\begin{cor} 
    If $G$ is not nilpotent then for some $p$-prime $ \mathcal{V}(G)\leq \mathcal{V}(D_{2p})$.
\end{cor}

\begin{prop}
    Let $G = C_{p} \rtimes C_3$ with p prime, where $p\equiv1\pmod 6$, then

\begin{equation*}
  \mathcal{V}(G) =   \frac{p+1}{p+3}. 
\end{equation*}

\end{prop}

\begin{proof}
Since every group of prime order is simple, the only possible simple subgroups here have order $p$ or order $3$. Consider $|G| = 3p$, where $p$ is prime and $p \equiv 1 \pmod{6}$. By Sylow’s theorems:

\begin{itemize}
    \item The number of Sylow $p$-subgroups $n_p$ divides $3$ and satisfies $n_p \equiv 1 \pmod{p}$. Since $p > 3$, the only possibility is $n_p = 1$. Hence, there exists a unique subgroup of order $p$, which is therefore normal in $G$.
    
    \item The number of Sylow $3$-subgroups $n_3$ divides $p$ and satisfies $n_3 \equiv 1 \pmod{3}$. Thus, $n_3$ is either $1$ or $p$. If $n_3 = 1$, then the Sylow $3$-subgroup is normal, and since the Sylow $p$-subgroup is also normal, $G$ would be the direct product $C_p \times C_3$, hence cyclic and abelian, isomorphic to $C_{3p}$. However, since $p \equiv 1 \pmod{3}$, there exist automorphisms of $C_p$ of order $3$. Therefore, there exists a nontrivial action $C_3 \to \operatorname{Aut}(C_p)$, and by taking the corresponding nontrivial semidirect product, one obtains a nonabelian group $G$. In this case, there can be no normal Sylow $3$-subgroup, and thus $n_3 = p$.
\end{itemize}

Therefore, if the action is trivial, yielding the direct product $C_p \times C_3$, there is exactly one subgroup of order $p$ and one subgroup of order $3$, for a total of $2$ simple subgroups. If the action is nontrivial, there is one subgroup of order $p$ and $p$ subgroups of order $3$, for a total of $p+1$ simple subgroups. Since the hypothesis $p \equiv 1 \pmod{6}$ guarantees the existence of an action of order $3$, the nontrivial semidirect product exists, and for this nonabelian group $G = C_p \rtimes C_3$, the natural answer is $p+1$. The denominator is obtained by considering the trivial subgroups. 

\end{proof}

\begin{cor} 
    $\lim_{p \longrightarrow \infty } \mathcal{V}(C_{p} \rtimes C_3) = 1.$
\end{cor}

Similarly to the result above, it is possible to evaluate other classes of groups in terms of direct and semi-direct products by observing asymptotic patterns such as the one presented in the following proposition.

\begin{prop}
Let $G$ be a finite group with p-prime
\begin{itemize}
      \item if $G = C_{p} \rtimes D_{2p}$, then $$\mathcal{V}(G) =   \frac{2p+1}{3p+5};$$
      \item if $G = D_{2p} \times D_{2p}$, then $$  \mathcal{V}(G) =   \frac{p^2+3p+1}{3p^2+8p+9};$$
      \item if $ G= D_{2p^2} \times D_{2p^2}$, then 
      $$  \mathcal{V}(G) = \frac{p^4 + 2p^2 + p + 1}{4p^4 + 4p^3 + 12p^2 + 12p + 13}.$$
\end{itemize}
\end{prop}

\begin{proof}
The proof is analogous to the previous cases.
\end{proof}

\begin{cor}
{\,}
    \begin{itemize}
        \item $\lim_{p \longrightarrow \infty } \mathcal{V}(C_{p} \rtimes D_{2p}) = \frac{2}{3};$
        \item $\lim_{p \longrightarrow \infty } \mathcal{V}(D_{2p} \times D_{2p}) = \frac{1}{3};$
        \item $\lim_{p \longrightarrow \infty } \mathcal{V}(D_{2p^2} \times D_{2p^2}) = \frac{1}{4}.$
    \end{itemize}
\end{cor}

\section{About density and applications}

In \cite{TM}, some results concerning the density of the quotient function under study were presented. Later, in \cite{Lazorec}, additional density results were obtained and discussed. In both works, the fact that the functions considered are multiplicative when the orders of the groups are coprime played a fundamental role in obtaining the results. It then naturally arises the question of whether similar results exist in the case of quotient functions that are not multiplicative or when the orders of the groups are not coprime. In this direction, \cite{deandrade2025} presents a density result that advances this line of investigation, and, in a similar spirit, one may consider a general result that includes the function $\mathcal{V}$ for certain families of groups.

\begin{lem}\label{lem:subsum_interval_general}
Let $(x_n)_{n\ge 1}$ satisfy
\[
x_n>0,\quad x_n\to 0,\quad \sum_{n=1}^{\infty} x_n = +\infty.
\]
Then, for every $y>0$, there exists a (typically infinite) subset $I\subset\mathbb{N}$ such that
\[
\sum_{n\in I} x_n = y.
\]
In particular, the set of \emph{finite subsums}
\[
\left\{\sum_{n\in P} x_n :\ P\subset\mathbb{N},\ |P|<\infty\right\}
\]
is dense in $[0,\infty)$.
\end{lem}

\begin{proof}
See \cite{NiteckiCantorvalsMonthly}.
\end{proof}

\begin{theorem} \label{Teorema 3.1}
Let $\mathcal{F}$ be a family of finite groups indexed by prime numbers $p$. 
Suppose there exists a sequence $(x_n)_{n \geq 1}$ associated with $\mathcal{V}$ that each term $x_n$ is of the form
\[
x_n = -\ln \left(\frac{p_n + q}{p_n + r} \right),
\]
where $p_n$ denotes the $n$-th prime number, $q<r \in \mathbb{N}$. 
Then the set
\[
\left\{ \prod_{n \in P} \mathcal{V}(\mathcal{F}_{p_n}) 
\;\middle|\;
P \subset \mathbb{N}\setminus\{0\},\ |P| < \infty
\right\}
\]
is dense in $(0,1]$.
\end{theorem}

\begin{proof}

Suppose that there exists a sequence $(x_n)_{n \geq 1}$ associated with $\mathcal{V}$ such that, for $G \in \mathcal{F}$, one has

\begin{equation}\label{seq}
 x_n = -\ln( \mathcal{V}(G)) = -\ln \left(\frac{p_n + q}{p_n + r} \right),   
\end{equation}

\noindent{then, according to Lemma 4.1 in \cite{Lazorec} and Lemma \ref{lem:subsum_interval_general}, if \ref{seq} is positive, satisfies $\lim_{n \longrightarrow \infty} x_n = 0$, and the series $\sum_{n=1}^{\infty} x_n$ is divergent, then the set containing the sums of all finite subsequences of $(x_n)$ is dense in $[0,\infty)$}. In fact, we have $x_n > 0$ and $\lim_{n \longrightarrow \infty} x_n = 0$. Finally, it remains to verify that the series is divergent. Consider

\begin{align*}
 \sum_{n=1}^{\infty} x_{n} =  \sum_{n=1}^{\infty} -\ln \left(\frac{p_n + q}{p_n + r} \right) = \sum_{n=1}^{\infty} \ln \left(\frac{p_n + r}{p_n + q} \right), 
\end{align*}

\noindent{since $q<r$, we have}
\[
\ln\frac{p_n+r}{p_n+q}
=\ln\!\left(1+\frac{r-q}{p_n+q}\right)
\ge \frac{r-q}{p_n+r}.
\]

Therefore,

$$
\sum_{n=1}^{\infty}\ln\frac{p_n+r}{p_n+q}
\ge (r-q)\sum_{n=1}^{\infty}\frac{1}{p_n+r}.
$$

Since $p_n + r \sim p_n$ and $\sum \frac{1}{p} = \infty$ (see \cite{Euler1737}, \cite{Niven1971}), the series diverges. Therefore, one has

   \[
   \overline{\left\{-\ln\prod_{n \in P} \left(\frac{p_n + q}{p_n + r} \right) \;\middle|\; P \subset \mathbb{N} \setminus \{0\}, |P| < \infty \right\}} = [0, \infty).
   \]

The exponential function \(e^{-x}\) is continuous, which preserves the density in the set. Then,

 \[
   \overline{\left\{\prod_{n \in P} \left(\frac{p_n + q}{p_n + r} \right) \;\middle|\; P \subset \mathbb{N} \setminus \{0\}, |P| < \infty \right\}} = (0, 1].
   \]

   Therefore, we have $\exp(-x_n)=\mathcal{V}(\mathcal{F}_{p_n})=\dfrac{p_n+q}{p_n+r}$ for each $n$, and hence
\[
\overline{\left\{\prod_{n\in P}\mathcal{V}(\mathcal{F}_{p_n})
\;\middle|\; P\subset\mathbb N\setminus\{0\},\ |P|<\infty\right\}}=(0,1],
\]
   
which proves the theorem.

\end{proof}

\begin{ex}
    Consider $\mathcal{F} = \{ C_{p}^{2}, D_{2p}\}$, for any $G \in \mathcal{F}$, we have $\mathcal{V}(G) = \left(\frac{p + 1}{p + 3} \right)$, where p is prime. Taking $q = 1$ and $r = 3$, the conditions of Theorem \ref{Teorema 3.1} are satisfied. Therefore, $\prod_{n \in P} \mathcal{V}(G)$ is dense in $(0,1]$.
\end{ex}

\begin{ex} 
Consider the family $\mathcal{F}=\{C_p^{2}, D_{2p}\}$, where $p$ is prime. As discussed above, for any
$G\in\mathcal{F}$ we have
\[
\mathcal{V}(G)=\frac{p+1}{p+3}.
\]
Given a target $t\in(0,1]$, we define $y:=-\ln(t)$ and consider
\[
x(p):=-\ln\!\left(\frac{p+1}{p+3}\right)
= \ln\!\left(\frac{p+3}{p+1}\right) > 0,
\qquad \lim_{p\to\infty} x(p)=0.
\]
By Theorem~\ref{Teorema 3.1}, finite products of the form
\[
\prod_{p\in P} \mathcal{V}(\mathcal{F}_p)
\]
are dense in $(0,1]$. Below we illustrate this mechanism by approximating three classical constants:
the Euler--Mascheroni constant $\gamma$ \cite{euler1740progressionibus},
the Meissel--Mertens constant $M$ \cite{mertens1874beitrag},
and the Landau--Ramanujan constant $K$ \cite{landau1908einteilung}.
In each case, we first choose a small set of primes (a greedy selection; cf.\ \cite{cormen2022introduction})
so that the partial product is close to the target, and then perform a \emph{fine adjustment} using a single large prime,
whose factor $\frac{p+1}{p+3}$ is very close to $1$.

\medskip
\noindent\textbf{(a) Euler–Mascheroni Constant.}
Let
\[
t=\gamma \approx 0.577215664901532860606512.
\]
In the greedy step we choose
\[
P=\{2,53,853\}:\quad
\frac{3}{5},\ \frac{54}{56}=\frac{27}{28},\ \frac{854}{856}=\frac{427}{428},
\]
hence
\[
\prod_{p\in P}\frac{p+1}{p+3}
=\frac{3}{5}\cdot\frac{27}{28}\cdot\frac{427}{428}
=\frac{4941}{8560}
\approx 0.577219626168224,
\]
with error
\[
\left|\frac{4941}{8560}-\gamma\right|
\approx 3.9612666914\cdot 10^{-6}.
\]
For the fine adjustment we take the prime \(p=291437\), obtaining
\[
\frac{p+1}{p+3}
=\frac{291438}{291440}
=\frac{145719}{145720}
\approx 0.999993135,
\]
and the refined product is
\[
\frac{4941}{8560}\cdot\frac{145719}{145720}
\approx 0.577215665012404,
\quad
\left|\,0.577215665012404-\gamma\,\right|
\approx 1.1087086804\cdot 10^{-10}.
\]

\medskip
\noindent\textbf{(b) Meissel--Mertens Constant.}
Let
\[
t=M \approx 0{.}261497212847642783755000.
\]
In the greedy step we choose
\[
P=\{2,3,5,13,521\}:\quad
\frac{3}{5},\ \frac{4}{6}=\frac{2}{3},\ \frac{6}{8}=\frac{3}{4},\ \frac{14}{16}=\frac{7}{8},\ 
\frac{522}{524}=\frac{261}{262}.
\]
hence
\[
\prod_{p\in P}\frac{p+1}{p+3}
=\frac{3}{5}\cdot\frac{2}{3}\cdot\frac{3}{4}\cdot\frac{7}{8}\cdot\frac{261}{262}
=\frac{5481}{20960}
\approx 0{.}261498091603053,
\]
with error
\[
\left|\frac{5481}{20960}-M\right|\approx 8{.}7875541066\cdot 10^{-7}.
\]
For the fine adjustment we take the prime  \(p=595157\), obtaining
\[
\frac{p+1}{p+3}=\frac{595158}{595160}=\frac{297579}{297580},
\]
and the final product is
\[
\frac{5481}{20960}\cdot\frac{297579}{297580}
\approx 0{.}261497212854174,
\quad
\left|\,0{.}261497212854174-M\,\right|
\approx 6{.}5309979647\cdot 10^{-12}.
\]

\medskip
\noindent\textbf{(c) Landau--Ramanujan Constant.}
Let
\[
t=K \approx 0{.}764223653589220662990699.
\]
In the greedy step we choose
\[
P=\{7,43,1543\}:\quad
\frac{8}{10}=\frac{4}{5},\ \frac{44}{46}=\frac{22}{23},\ \frac{1544}{1546}=\frac{772}{773}.
\]
Therefore,
\[
\prod_{p\in P}\frac{p+1}{p+3}
=\frac{4}{5}\cdot\frac{22}{23}\cdot\frac{772}{773}
=\frac{67936}{88895}
\approx 0{.}764227459362169,
\]
with error
\[
\left|\frac{67936}{88895}-K\right|\approx 3{.}8057729482\cdot 10^{-6}.
\]
For the fine adjustment we take the prime \(p=401627\), obtaining
\[
\frac{p+1}{p+3}=\frac{401628}{401630}=\frac{200814}{200815},
\]
and the refined product is
\[
\frac{67936}{88895}\cdot\frac{200814}{200815}
\approx 0{.}764223653732812,
\quad
\left|\,0{.}764223653732812-K\,\right|
\approx 1{.}4359102796\cdot 10^{-10}.
\]

\medskip
The fractions above are short enough to allow direct manual verification. They illustrate, in a concrete way,
the sum-product mechanism underlying Theorem~\ref{Teorema 3.1}: by choosing a finite set of primes,
we make \(\sum x(p)\) approximate \(y=-\ln(t)\) and, consequently,
\(\prod \mathcal{V}(\mathcal{F}_p)=\exp(-\sum x(p))\) approximate \(t\).

\end{ex}

\begin{cor}\label{cor:fine-tune}
Under the hypotheses of Theorem~\ref{Teorema 3.1}, fix $t \in (0,1]$ and set $y := -\ln(t)$.
Assume that $P_0$ is a finite set of indices such that
\[
S_0 := \sum_{n \in P_0} x_n \le y,
\qquad \Delta := y - S_0 > 0.
\]
Then, for every $\eta \in (0, \Delta)$, there exists a finite set $P \supset P_0$ such that
\[
\left| \sum_{n \in P} x_n - y \right| < \eta,
\quad \text{and consequently} \quad
\left| \prod_{n \in P} \mathcal{V}(\mathcal{F}_{p_n}) - t \right|
\le t \, (e^{\eta} - 1).
\]
\end{cor}

\begin{proof}
Fix $\eta \in (0, \Delta)$. Since
\[
x_n = \ln\!\left(1 + \frac{r - q}{p_n + q}\right) \to 0
\quad \text{as } n \to \infty,
\]
there exists $N$ such that $x_n < \eta$ for all $n \ge N$.
Consider the tail series $\sum_{n \ge N,\; n \notin P_0} x_n$, which diverges.
Hence, there exists a finite set
\[
Q \subset \{ n \ge N : n \notin P_0 \},
\]
such that
\[
S_0 + \sum_{n \in Q} x_n > y.
\]
Choose $Q$ minimal with respect to this property.
Define $P := P_0 \cup Q$ and $S := \sum_{n \in P} x_n$.
Let $m$ be the last index added, that is, the largest element of $Q$.
By minimality, we have $S - x_m \le y < S$, and therefore
\[
0 < S - y \le x_m < \eta,
\]
that is, $|S - y| < \eta$.
Finally, since $\prod_{n \in P} \mathcal{V}(\mathcal{F}_{p_n}) = \exp(-S)$ and $t = \exp(-y)$,
\[
\left|\exp(-S)-\exp(-y)\right|
= e^{-y}\left|e^{-(S-y)}-1\right|
\le e^{-y}\left(e^{|S-y|}-1\right)
\le t\,(e^{\eta}-1).
\]
\end{proof}


\medskip
Next, we record consequences and extensions of Theorem ~\ref{Teorema 3.1}.

\begin{cor}[Application to the logarithmic sequence of $\mathcal{V}$]
\label{cor:V_subsum}
Under the hypotheses of Theorem~\ref{Teorema 3.1}, define
\[
x_n:=-\ln\!\big(\mathcal{V}(\mathcal{F}_{p_n})\big)
=-\ln\!\left(\frac{p_n+q}{p_n+r}\right)
=\ln\!\left(\frac{p_n+r}{p_n+q}\right).
\]
Then, for every $y>0$, there exists $I\subset\mathbb N$ such that
\[
\sum_{n\in I} -\ln\!\big(\mathcal{V}(\mathcal{F}_{p_n})\big)=y.
\]
In particular, for every $t\in(0,1)$ there exists $I\subset\mathbb N$ such that the infinite product converges and
\[
\prod_{n\in I}\mathcal{V}(\mathcal{F}_{p_n})=t.
\]
\end{cor}

\begin{proof}
In Theorem~\ref{Teorema 3.1} it was already verified that $(x_n)$ satisfies \(x_n>0\), \(x_n\to0\), and \(\sum x_n=+\infty\).
Hence, by Lemma~\ref{lem:subsum_interval_general}, given $y>0$ there exists $I$ such that \(\sum_{n\in I}x_n=y\).

For $t\in(0,1)$, take $y=-\ln(t)>0$; then
\[
t=e^{-y}
=\exp\!\left(-\sum_{n\in I}x_n\right)
=\prod_{n\in I} e^{-x_n}
=\prod_{n\in I}\mathcal{V}(\mathcal{F}_{p_n}).
\]
The convergence of the product follows from the convergence of \(\sum_{n\in I}x_n=y\).
\end{proof}

\begin{theorem}[Density restricting to primes in an arithmetic progression]
\label{thm:AP_dense}
Fix $m\ge 2$ and $a$ with $\gcd(a,m)=1$. Let $(p_{n_k})_{k\ge1}$ be the subsequence of primes satisfying
\[
p_{n_k}\equiv a\pmod m.
\]
Then the set of finite products
\[
\left\{
\prod_{k\in P}\mathcal{V}(\mathcal{F}_{p_{n_k}})
\;\middle|\;
P\subset\mathbb N,\ |P|<\infty
\right\}
\]
is dense in $(0,1]$.
\end{theorem}

\begin{proof}
Define, for $k\ge1$,
\[
x_k:=-\ln\!\left(\mathcal{V}(\mathcal{F}_{p_{n_k}})\right)
=\ln\!\left(1+\frac{r-q}{p_{n_k}+q}\right)>0.
\]
We have \(x_k\to0\) as \(k\to\infty\). Using the inequality $\ln(1+u)\ge \frac{u}{1+u}$ for $u>0$, we obtain
\[
x_k
\ge \frac{r-q}{p_{n_k}+r}.
\]
Hence,
\[
\sum_{k=1}^{\infty}x_k
\ge
(r-q)\sum_{k=1}^{\infty}\frac{1}{p_{n_k}+r}.
\]

Since $p_{n_k}+r\approx p_{n_k}$, it suffices to know that $\sum_k 1/p_{n_k}$ diverges. For this, we use a Mertens-type result for arithmetic progressions (see \cite{Williams1974MertensAP}), which implies
\[
\sum_{\substack{p\le x\\ p\equiv a\ (m)}} \frac1p
=
\frac{1}{\varphi(m)}\log\log x + O(1),
\qquad x\to\infty,
\]
and therefore
\[
\sum_{p\equiv a\ (m)}\frac1p
\]
diverges.

Thus, $(x_k)$ is positive, tends to $0$, and has divergent sum. By Lemma~\ref{lem:subsum_interval_general},
the finite subsums of $(x_k)$ are dense in $[0,\infty)$; applying $\exp(-\cdot)$, we obtain the density of the finite products
\[
\prod_{k\in P}\mathcal{V}(\mathcal{F}_{p_{n_k}})
=
\exp\!\left(-\sum_{k\in P}x_k\right)
\]
in $(0,1]$.
\end{proof}

\begin{ex}[Example of the Theorem~\ref{thm:AP_dense} with the Catalan Constant]\label{ex:AP_dense_Catalan}
Consider the family $\mathcal{F}=\{C_p^2,D_{2p}\}$, for which
\[
\mathcal{V}(\mathcal{F}_p)=\frac{p+1}{p+3}\qquad (q=1,\ r=3).
\]
Fix the arithmetic progression of primes with
\[
m=4,\qquad a=1,\qquad \gcd(a,m)=1,
\]
That is, we only allow primes $p\equiv 1\pmod 4$.

Let $G$ be \emph{Catalan Constant}
\cite{Zudilin2003AperyLikeCatalan},
\[
G=\sum_{n=0}^{\infty}\frac{(-1)^n}{(2n+1)^2}\approx 0.9159655941772502.
\]
A constructive procedure (of the greedy type \cite{cormen2022introduction} in the logarithmic domain, restricted to the progression $p\equiv 1\pmod 4$)
produces the finite set of primes
\[
P=\{29,\,89,\,1597,\,241973\}\subset\{p\ \text{prime}:\ p\equiv 1\ (\mathrm{mod}\ 4)\}.
\]

With this set $P$, we obtain the finite product
\[
\prod_{p\in P}\mathcal{V}(\mathcal{F}_p)
=\prod_{p\in P}\frac{p+1}{p+3}
=\frac{30}{32}\cdot\frac{90}{92}\cdot\frac{1598}{1600}\cdot\frac{241974}{241976}
\approx 0.9159655949838855.
\]
Therefore,
\[
\left|\prod_{p\in P}\frac{p+1}{p+3}-G\right|
\approx 8.06635314098969\times 10^{-10}<10^{-9}.
\]

This example shows that the density guaranteed in Theorem~\ref{thm:AP_dense} remains valid even under the restriction
to a single residue class of primes (here, $p\equiv 1\pmod 4$), allowing one to approximate specific constants by finite products
composed solely of factors arising from primes in the arithmetic progression.
Moreover, the Mertens-type result in arithmetic progressions used in the proof may be viewed as the logarithmic counterpart
of the product version of Mertens' third theorem, now restricted to primes satisfying $p\equiv a\pmod m$.

Thus, $G$ is approximated (with error $<10^{-9}$) by a finite product formed \emph{only} from primes
$p\equiv 1\pmod 4$, in accordance with the conclusion of Theorem~\ref{thm:AP_dense}.
\end{ex}

\section{Relationships with other functions}

As already seen, if $G$ is solvable, the only simple groups that $G$ admits are cyclic subgroups of prime order. Therefore, it is possible to establish the following order relation, considering the function under study as a restriction on the degree of cyclicity in solvable groups.

\begin{equation} \label{relacao}
    \mathcal{V}(G) \leq cdeg(G) \leq \mathfrak{J}(G).
\end{equation}

This order relationship makes it possible, for example, to study families of groups in terms of their degree of cyclicity, which had not previously been investigated for this purpose.

\begin{prop}
    Consider the family $\{G_n\}_{n \ge 0},$ defined in \ref{Def. 1}. Let $G_0 = (C_p \times C_p) \rtimes C_p = C_{p}^{2}\rtimes C_p $, then $\lim_{p \longrightarrow \infty} cdeg(G_0) = 1.$
\end{prop}

\begin{proof}
Since $G$ is a finite $p$-group, it follows that $\mathfrak{J}(G_0) = 1$. Using \cite{GAP4}, we have

 $$\mathcal{V}(G_0) = \frac{p^2+p+1}{p^2+2p+4},$$

 \noindent{considering the relation \ref{relacao}, it follows}

\begin{align*}
   \lim_{p \longrightarrow \infty}  \mathcal{V}(G_0) & \leq \lim_{p \longrightarrow \infty} cdeg(G_0) \leq \lim_{p \longrightarrow \infty} \mathfrak{J}(G_0).\\
\end{align*}

\noindent{Therefore, $\lim_{p \longrightarrow \infty} cdeg(G_0) = 1.$}
\end{proof}

The order relation~\ref{relacao} is well defined in view of the functions involved. A nontrivial issue, however, is the identification of families of finite groups for which equality holds. In~\cite{deandrade2025}, a family of groups is exhibited for which the equality $cdeg(G)=\mathfrak{J}(G)$ is satisfied. This naturally leads to the question of whether other functions — whose denominators involve the group order rather than the number of subgroups, such as those considered in~\cite{Garonzi_2018} and~\cite{Lazorec} — also respect this order relation. In this context, the most natural approach is to determine the families of groups for which $|G|=|L(G)|$. The following proposition provides a partial answer to this question.  

\begin{prop} \label{prop. 4.2}
Let $G = D_{2^{a+1}p}$. If $p = 2^{a+1} + 2a + 1$ is prime, then
\begin{equation*}
    |L(G)| = |G| = 2^{a+1}p.
\end{equation*}
\end{prop}

\begin{proof}

We want to check when 

\begin{equation}\label{eq. 4}
\tau(m) + \sigma(m) = 2m     
\end{equation}

Suppose $m=2^a p$ with $a \geq 1$ and $p$ an odd prime. Then, by multiplicativity,

$$
\tau(m)= 2(a+1),\quad
\sigma(m)=\sigma(2^a)\sigma(p)=(2^{a+1}-1)(1+p).
$$

The equation $\tau(m)+\sigma(m)=2m=2^{a+1}p$ becomes

$$
2(a+1)+(2^{a+1}-1)(1+p)=2^{a+1}p.
$$

Rearranging gives us

$$
p=2^{a+1}+2a+1.
$$

Therefore, $m=2^ap$ is a solution if, and only if, $p=2^{a+1}+2a+1$ is prime.
\end{proof}

\begin{rem}
It is worth noting that this solution is not unique to dihedral groups. There is another order configuration that also seems to satisfy the equality \ref{eq. 4}. For $D_{2n}$, with $2n = 2^{j}pq,$ where $j = \{2,4,8,...\}$ and $p$ and $q$ are distinct primes. Apparently there are no representatives for $j=32,128$.    
\end{rem}
    
For the dihedral groups of Proposition \ref{prop. 4.2}, the following relation holds

\begin{equation}\label{relação ordem 8} 
    \mathcal{V}(G) < cdeg(G) = \alpha(G) <\mathfrak{J}(G) < \beta(G) =1,
\end{equation}

\noindent{where, the function $\alpha$ is defined in \cite{Garonzi_2018}}.

\section{Further research}

Given what has already been established about solvable groups, it becomes natural to study groups that are not solvable. From these groups, it is possible to investigate the asymptotic behavior of other families that are neither simple nor solvable, such as those described in Proposition \ref{prop. 2.4}. At this stage, the use of GAP \cite{GAP4} proved to be essential, despite presenting considerable limitations in the case of simple groups of very large order. In Table \ref{Tab. 1}, one can observe some families of simple groups for which $\mathcal{V}$ has been computed. Patterns are observed in two families: $PSL(2,p)$ and $A_n$, for prime $p$ and $n \geq 5$, respectively. For $A_n$, it is possible to formulate the following proposition, since the observed decreasing pattern:

\begin{prop} \label{prop. 5.1}
    Let $G = A_n$, for $n\geq5$, then
    \begin{equation*}
        \lim_{n \longrightarrow \infty}\mathcal{V}(G) =0.
    \end{equation*}
\end{prop}

To prove this proposition, we will need the following lemma.

\begin{lem}\label{lemma 5.2}
For $n \geq 5$, one has
\[
\lim_{n \to \infty} \mathcal{V}(S_n)=0.
\]
\end{lem}

\begin{proof}
It suffices to take the limit
\[
\lim_{n \to \infty} \frac{4.89n + 1141.33}{2^{\frac{n^2}{16+o(n^2)}}}=0,
\]
where the denominator is given in \cite{RoneyDougal2025SubgroupsOS} and the numerator is given in \cite{Borovik1996Maximal}.
\end{proof}

We may now prove Proposition \ref{prop. 5.1}.

\begin{proof}
Note that
\[
0< \mathcal{V}(A_n)\leq \mathcal{V}(S_n).
\]

Taking limits on both sides and applying Lemma \ref{lemma 5.2}, the result follows.
\end{proof}

Another interesting question is derived from Proposition \ref{prop. 4.2}, arising from asking whether $2^{a+1} + 2a + 1$ is prime infinitely often. A direct implication would be the existence of infinitely many dihedral groups for which the order relation \ref{relação ordem 8} holds.

\newpage

\bibliographystyle{plain}  
\bibliography{bibliography}  

@article{Garonzi_2018,
   title={On the Number of Cyclic Subgroups of a Finite Group},
   volume={49},
   ISSN={1678-7714},
   url={http://dx.doi.org/10.1007/s00574-018-0068-x},
   DOI={10.1007/s00574-018-0068-x},
   number={3},
   journal={Bulletin of the Brazilian Mathematical Society, New Series},
   publisher={Springer Science and Business Media LLC},
   author={Garonzi, Martino and Lima, Igor},
   year={2018},
   month=jan, pages={515–530} }

@article{Lazorec,
author = {Lazorec, Mihai-Silviu},
title = {A connection between the number of subgroups and the order of a finite group},
journal = {Journal of Algebra and Its Applications},
volume = {21},
number = {01},
pages = {2250001},
year = {2022},
doi = {10.1142/S0219498822500013},

URL = { 
    
        https://doi.org/10.1142/S0219498822500013
    
    

},
eprint = { 
    
        https://doi.org/10.1142/S0219498822500013
    
    

}

}

@manual{GAP4,
  author       = {The GAP Group},
  title        = {GAP -- Groups, Algorithms, and Programming},
  year         = {2022},
  version      = {4.12.2},
  note         = {\url{https://www.gap-system.org}}
}

@article{Marius,
author = {Tarnauceanu, Marius and Tóth, László},
year = {2015},
month = {04},
pages = {489-504},
title = {Cyclicity degrees of finite groups},
volume = {145},
journal = {Acta Mathematica Hungarica},
doi = {10.1007/s10474-015-0480-2}
}

@article{TM,
author = {Tarnauceanu, Marius},
year = {2017},
month = {01},
pages = {115-126},
title = {Normality degrees of finite groups},
volume = {33}
}

@misc{deandrade2025,
      title={Proportion of Nilpotent Subgroups in Finite Groups and Their Properties}, 
      author={João Victor M. de Andrade and Leonardo Santos da Cruz},
      year={2025},
      eprint={2501.11724},
      archivePrefix={arXiv},
      primaryClass={math.GR},
      url={https://arxiv.org/abs/2501.11724}, 
}

@article{It1951, title={On the Degrees of Irreducible Representations of a Finite Group}, volume={3}, DOI={10.1017/S0027763000012162}, journal={Nagoya Mathematical Journal}, author={Itô, Noboru}, year={1951}, pages={5–6}}

@article{Gustafson01111973,
author = {W. H. Gustafson},
title = {What is the Probability that Two Group Elements Commute?},
journal = {The American Mathematical Monthly},
volume = {80},
number = {9},
pages = {1031--1034},
year = {1973},
publisher = {Taylor \& Francis},
doi = {10.1080/00029890.1973.11993437},


URL = { 
    
        https://doi.org/10.1080/00029890.1973.11993437
    
    

},
eprint = { 
    
        https://doi.org/10.1080/00029890.1973.11993437
    
    

}

}

@article{Betz16032022,
author = {Alexander Betz and David A. Nash},
title = {Classifying Groups With a Small Number of Subgroups},
journal = {The American Mathematical Monthly},
volume = {129},
number = {3},
pages = {255--267},
year = {2022},
publisher = {Taylor \& Francis},
doi = {10.1080/00029890.2022.2010493},


URL = { 
    
        https://doi.org/10.1080/00029890.2022.2010493
    
    

},
eprint = { 
    
        https://doi.org/10.1080/00029890.2022.2010493
    
    

}

}

@article{Euler1737,
  author    = {Leonhard Euler},
  title     = {Variae observationes circa series infinitas},
  journal   = {Commentarii Academiae Scientiarum Petropolitanae},
  volume    = {9},
  pages     = {160--188},
  year      = {1744},
  note      = {Escrito em 1737},
  url       = {https://eulerarchive.maa.org/pages/E072.html}
}

@article{Niven1971,
  author    = {Ivan Niven},
  title     = {A Proof of the Divergence $\Sigma\, 1/p$},
  journal   = {The American Mathematical Monthly},
  volume    = {78},
  number    = {3},
  pages     = {272--273},
  year      = {1971},
  doi       = {10.1080/00029890.1971.11992740},
  note      = {Prova curta da divergência da soma dos recíprocos dos primos}
}

@article{He2022SomeIO,
  title={Some identities of Gaussian binomial coefficients},
  author={T. X. He and Anthony G. Shannon and Peter J.-S. Shiue},
  journal={Annales Mathematicae et Informaticae},
  year={2022},
  url={https://api.semanticscholar.org/CorpusID:245689346}
}

@article{euler1740progressionibus,
  author  = {Euler, L.},
  title   = {De progressionibus harmonicis observationes},
  journal = {Commentarii academiae scientiarum Petropolitanae},
  volume  = {7},
  year    = {1740},
  pages   = {150--156},
  note    = {(1734--35)}
}

@article{mertens1874beitrag,
  author  = {Mertens, F.},
  title   = {Ein Beitrag zur analytischen Zahlentheorie},
  journal = {Journal f{\"u}r die reine und angewandte Mathematik},
  volume  = {78},
  year    = {1874},
  pages   = {46--62}
}

@article{landau1908einteilung,
  author  = {Landau, E.},
  title   = {\"{U}ber die Einteilung der positiven ganzen Zahlen in vier Klassen nach der Mindestzahl der zu ihrer additiven Zusammensetzung erforderlichen Quadrate},
  journal = {Archiv der Mathematik und Physik},
  volume  = {13},
  year    = {1908},
  pages   = {305--312}
}

@book{cormen2022introduction,
  author    = {Cormen, Thomas H. and Leiserson, Charles E. and Rivest, Ronald L. and Stein, Clifford},
  title     = {Introduction to Algorithms},
  edition   = {4},
  publisher = {MIT Press},
  year      = {2022}
}

@article{NiteckiCantorvalsMonthly,
  author  = {Nitecki, Zbigniew},
  title   = {Cantorvals and Subsum Sets of Null Sequences},
  journal = {The American Mathematical Monthly},
  volume  = {122},
  number  = {9},
  year    = {2015},
  pages   = {862--870},
  doi     = {10.4169/amer.math.monthly.122.9.862}
}

@article{Williams1974MertensAP,
  author  = {Kenneth S. Williams},
  title   = {Mertens' Theorem for Arithmetic Progressions},
  journal = {Journal of Number Theory},
  volume  = {6},
  year    = {1974},
  pages   = {353--359},
  url     = {https://people.math.carleton.ca/~williams/papers/pdf/057.pdf}
}

@article{Zudilin2003AperyLikeCatalan,
  author  = {Zudilin, W.},
  title   = {An Ap\'ery-like Difference Equation for Catalan's Constant},
  journal = {The Electronic Journal of Combinatorics},
  volume  = {10},
  number  = {1},
  pages   = {R14},
  year    = {2003},
  doi     = {10.37236/1707},
  url     = {https://www.combinatorics.org/ojs/index.php/eljc/article/view/v10i1r14}
}

@inproceedings{RoneyDougal2025SubgroupsOS,
  title={Subgroups of symmetric groups: enumeration and asymptotic properties},
  author={Colva M. Roney‐Dougal and Gareth Tracey},
  year={2025},
  url={https://api.semanticscholar.org/CorpusID:276885489}
}

@article{Borovik1996Maximal,
  author    = {Borovik, Alexandre V. and et al.},
  title     = {Maximal Subgroups in Finite and Profinite Groups},
  journal   = {Transactions of the American Mathematical Society},
  volume    = {348},
  number    = {9},
  pages     = {3745--3761},
  year      = {1996},
  publisher = {American Mathematical Society},
  url       = {http://www.jstor.org/stable/2155252},
  note      = {Accessed: 2026-05-07}
}

\newpage
\section{Appendix I}\label{Apendice 2} 

\begin{table}[ht]
\centering
\caption{Simple evaluated groups}
\begin{tabular}{cccccc}
  \hline
 $G$ & $|G|$ & $s(G)$ & $|L(G)|$ & $V(G)$ \\ 
  \hline
 $A_5$ & 60 & 32 & 59 & 0.54237 \\ 
 $A_6$ & 360 & 134 & 501 & 0.26747 \\ 
 $A_7$ & 2520 & 627 & 3786 & 0.16561 \\ 
 $A_8$ & 20160 & 2848 & 48337 & 0.05892 \\ 
 $A_9$ & 181440 & 12443 & 508402 & 0.02447 \\ 
 $A_{10}$ & 1814400 & 116109 & 6469142 & 0.01795 \\ 
 $A_{11}$ & 19958400 & 1666154 & 81711572 & 0.02039 \\ 
 $A_{12}$ & 239500800 & 19272177 & 2019160542 & 0.00954 \\ 
 $A_{13}$ & 3113510400 & 188697242 & 31945830446 & 0.00591 \\ 
 PSL(3,2) & 168 & 58 & 179 & 0.32402 \\ 
 PSL(2,11) & 660 & 211 & 620 & 0.34032 \\ 
 PSL(2,13) & 1092 & 275 & 942 & 0.29193 \\ 
 PSL(2,17) & 2448 & 308 & 2420 & 0.12727 \\ 
 PSL(2,19) & 3420 & 667 & 2912 & 0.22905 \\ 
 PSL(2,23) & 6072 & 807 & 5915 & 0.13643 \\ 
 PSL(2,29) & 12180 & 2119 & 10040 & 0.21106 \\ 
 PSL(2,31) & 14880 & 1986 & 15413 & 0.12885 \\ 
 PSL(2,37) & 25308 & 2111 & 17731 & 0.11906 \\ 
 PSL(2,41) & 34440 & 4553 & 36129 & 0.12602 \\ 
 PSL(2,43) & 39732 & 3743 & 25462 & 0.14700 \\ 
 PSU(3,3) & 6048 & 752 & 5150 & 0.14602 \\ 
 PSU(3,4) & 62400 & 5540 & 31373 & 0.17658 \\ 
 PSU(3,5) & 126000 & 18257 & 179308 & 0.10182 \\ 
 PSU(4,3) & 3265920 & 916160 & 10009764 & 0.09153 \\ 
 PSU(3,8) & 5515776 & 273032 & 6864058 & 0.03978 \\ 
 Sz(8) & 29120 & 4552 & 17295 & 0.26320 \\ 
 Sz(32) & 32537600 & 1080352 & 21170191 & 0.05103 \\ 
 $M_{11}$ & 7920 & 1147 & 8651 & 0.13259 \\ 
 $M_{12}$ & 95040 & 11204 & 214871 & 0.05214 \\ 
 $M_{22}$ & 443520 & 111554 & 941627 & 0.11847 \\ 
 $M_{23}$ & 10200960 & 1273280 & 17318406 & 0.07352 \\ 
 $M_{24}$ & 244823040 & 16219100 & 1363957253 & 0.01189 \\ 
 $J_1$ & 175560 & 19287 & 158485 & 0.12170 \\ 
 $J_2$ & 604800 & 46800 & 1104344 & 0.04238 \\ 
 $J_3$ & 50232960 & 4941750 & 71564248 & 0.06905 \\ 
 HS & 44352000 & 4640440 & 149985646 & 0.03094 \\ 
 McL & 898128000 & 118406000 & 1719739392 & 0.06885 \\ 
 He & 4030387200 & 121547000 & 22303017686 & 0.00545 \\ 
   \hline
\end{tabular}
\label{Tab. 1}
\end{table}

\end{document}